\documentclass[11pt]{article}
\usepackage{enumerate}
\usepackage{url}
\usepackage{color}
\usepackage[textwidth=0.75\paperwidth,textheight=0.75\paperheight,centering]{geometry}
\usepackage{graphicx}
\usepackage{float}
\usepackage{amsmath}
\usepackage{amsfonts}
\usepackage{amssymb}
\usepackage{amsthm}
\usepackage{ulem}

\theoremstyle{plain}
\newtheorem{theorem}{Theorem}
\newtheorem{corollary}[theorem]{Corollary}
\newtheorem{proposition}[theorem]{Proposition}
\newtheorem{lemma}[theorem]{Lemma}
\theoremstyle{definition}

\newtheorem{example}[theorem]{Example}
\newtheorem{remark}[theorem]{Remark}

\usepackage{natbib}
\setcitestyle{comma}

\begin{document}

\title{New Exponential Dispersion Models for Count Data: The ABM 
and LM Classes}
\author{Shaul K. Bar-Lev\thanks{%
Faculty of Technology Management,\ Holon Institute of Technology, Holon,
Israel; email: shaulb@hit.ac.il } \and Ad Ridder\thanks{%
School of Business and Economics, Vrije University of
Amsterdam, Amsterdam, The Netherlands; email: ad.ridder@vu.nl }}
\maketitle

\begin{abstract}
In their fundamental paper on cubic variance functions (VFs), Letac and Mora
(\textit{The Annals of Statistics},1990) presented a systematic, rigorous
and comprehensive study of natural exponential families (NEFs) on the real
line, their characterization through their VFs and mean value
parameterization. They presented a section that for some reason has been
left unnoticed. This section deals with the construction of VFs associated
with NEFs of counting distributions on the set of nonnegative integers and
allows to find the corresponding generating measures. As EDMs are based on
NEFs, we introduce in this paper two 
new classes of EDMs based on their results. 
For these classes, which are associated with simple VFs, we derive
their mean value parameterization and their associated generating measures.
We also prove that they have some desirable properties. 
Both classes are 
shown to be overdispersed and zero inflated in ascending order,
making them as competitive statistical models for those in use in both,
statistical and actuarial modeling. To our best knowledge, 
the 
classes of counting distributions we present in this paper, have not been
introduced or discussed  before in the literature. 
To show that our classes can serve as
competitive statistical models for those in use (e.g., Poisson, Negative
binomial),
we include 
a numerical example of real data. 
In this example, we  compare the performance of 
our classes with relevant competitive models. 

\end{abstract}

\section{Introduction and Background}\label{s:intro}
Natural exponential families (NEFs) and exponential dispersion models
(EDMs)\ on $\mathbb{R}$ play an important role both in probability and
statistical applications. Most of the frequently used distributions are
indeed belonging to such models. However, a huge number of NEFs (or EDMs)
have not been used in probabilistic or statistical modelling for two main
reasons: they have not been revealed or do not have explicit functional
forms (even not via power series expansions). This, despite the fact that
they could have provided significant and new models useful in statistical
applications. Indeed, the main purpose of this paper is to expose the
statistical research community to various classes of such NEFs. A thorough
discussion on this observation is presented in \citet{barlev2017}.

One of the most forsaken reference representing the above situation is the
fundamental paper \citet{letac1990} 
on NEFs
which provides a thorough description and analytic properties of such
families along with their mean value parameterization. In spite of the fact
that their article received many citations, a major and important part of
the article was somehow abandoned without being noticed. This part refers to
the section dealing with the construction of NEFs of counting distributions
on the set of nonnegative integers $\mathbb{N}_{0}$. These families are
represented by either polynomial variance functions (VFs) or other nice
forms. Moreover, in their Proposition 4.4 they explicitly present a formula
that allows to compute, at least numerically, the counting measure $\mu $
which generates the appropriate NEF in terms of its mean $m$ 
\citep[for further details see also][]{barlev2017}.
Such a formula requires
(except for  a few limited special cases) 
some rather cumbersome numerical
calculations of the $n$-th derivative of product of functions depending on
the mean $m$ which are needed for calculating the 
mass of $\mu $ at the point $n$.

In our opinion, one of the reasons why this formula as well as Proposition
4.4 of \citet{letac1990}, 
were not used is that in the eighties and nineties of the
last century (when the \citet{letac1990}
article was just published) is related to
the fact that there were no powerful mathematical programs that would allow
the complex and cumbersome calculations of the mass of $\mu $ on the
nonnegative integers. Fortunately, nowadays, the situation has changed and
existing powerful computing software are available and might be used to
calculate. However, despite the nowadays availability of existing powerful
software, it is still intricate or even not possible to carry out the
probability calculation of the relevant NEFs in their general settings. It
is therefore necessary to locate special cases of NEFs complying Proposition
4.4 of \citet{letac1990} for which the software application is possible. And
indeed, our aim in this paper is to achieve this goal and introduce 
two 
classes of NEFs, which through further mathematical improvements, allow
the calculation of  the appropriate count probabilities 
of these subclasses of NEFs. 
We need to point out here that locating such classes is not as simple
as it seems, and requires great care and thoughts in choosing them. To our
best knowledge,   the  classes 
of counting distributions we present in
this paper have not been introduced or discussed 
 before  in the literature. 
A fact that will lead to exposure of numerous counting NEFs (as
well as EDMs) that can serve as competitive statistical models for those in
use (e.g., Poisson, Negative binomial) in both, statistical and actuarial
modeling.

For this we need to present some preliminaries. As is well known, and as
will seen in the sequel, EDMs are based on NEFs. Hence, we first need to
present some basic properties of NEFs   and  VFs, 
mean value parameterization, and then EDMs.

Let $\mu $ be a positive Radon measure on $\mathbb{R}$ with convex support 
$C_{\mu }$. Consider the set
\begin{equation}
D_{\mu }\doteq \left\{ \theta \in \mathbb{R} : 
L_{\mu }(\theta )\doteq
\int\nolimits_{\mathbb{R}}\exp(\theta x)\mu (dx)
<\infty \right\} ,  \label{1}
\end{equation}
and assume that $\Theta_{\mu}\doteq \mathrm{int}\, D_{\mu}$ is nonempty. 
Then, the NEF $\mathcal{F}(\mu)$ 
generated by $\mu $ is defined by the set of probability
 distributions 
\begin{equation} 
\mathcal{F}(\mu) \doteq 
\Big\{ F\big(\theta ,\mu(dx)\big)=
\exp\big( \theta x-k_{\mu }(\theta) \big)
\mu (dx) : \theta \in\Theta_\mu \Big\} ,  \label{3}
\end{equation}
where $k_{\mu }(\theta) \doteq \log L_{\mu }(\theta )$ is the
cumulant transform of $\mu$; $k_{\mu }$ is strictly convex and real
analytic on $\Theta_\mu $.  Moreover,  
$k_{\mu}^{\prime }(\theta )$ and 
$k_{\mu}^{\prime \prime}(\theta )$, $\theta \in \Theta_\mu $, 
are the respective mean and variance
corresponding to $F(\theta ,\mu )$, and the open interval 
$M_\mu$  $\doteq k_{\mu}^{\prime}(\Theta_\mu)$ 
is called the mean 
domain of $\mathcal{F}(\mu)$. 

An important observation is that measure $\mu$ is not unique
for $\mathcal{F}(\mu)$. Let $\mathcal{M}$ be all Radon measures $\nu$
on $\mathbb{R}$ for which $L_\nu(\theta)<\infty$ on domain $\Theta_\nu$.
Consider two measures $\mu, \mu^*\in\mathcal{M}$, and suppose that
$\mu^*$ is an exponential shift of $\mu$; i.e., 
$\mu^*(dx)=e^{a+bx}\mu(dx)$ for some real $a,b$.
Then a simple calculation shows that 
$\mathcal{F}(\mu)=\mathcal{F}(\mu^*)$. This holds
also reversely, if $\mathcal{F}(\mu)=\mathcal{F}(\mu^*)$
for two measures $\mu, \mu^*\in\mathcal{M}$,
then one is an exponential shift of the other.
Consequently, we may denote the NEF by $\mathcal{F}=\mathcal{F}(\mu)$
and its the mean domain $M=M_\mu$ to stress that
these do not depend on $\mu$.

Since the 
function $k_\mu^{\prime}: \Theta_\mu\to M$ 
is one-to-one, its
inverse function
$\big(k_\mu^{\prime}\big)^{-1} $
$: M\to \Theta_\mu$ is well defined. 
When we compute the variance $V_\mu(\theta)\doteq k_\mu^{\prime\prime}(\theta)$
of the distribution $F(\theta,\mu)$ as a function of the mean $m\in M$, i.e.,
\begin{equation}
V_\mu(m) = 
k_\mu^{\prime \prime}\big((k_\mu^{\prime})^{-1}(m)\big),
\label{S2.2}
\end{equation}
it also does not depend on $\mu$, and we denote it by $V(m)$.
The map $m\mapsto V(m)$ from $M$ into $\mathbb{R}^{+}$ is called the
variance function (VF) of $\mathcal{F}$. In fact, a VF of an 
NEF $\mathcal{F}$ is a pair 
$(V, M)$ which uniquely determines the NEF within the class of NEFs 
\citep[see][]{morris1982,letac1990}.
It is important to emphasize that a VF is a transform,
not of a particular distribution, but rather of a family $\mathcal{F}$ 
in the sense that if two VFs $(V_{1}, M_{1})$ and $(V_{2}, M_{2})$ of two
NEFs $\mathcal{F}_{1}$ and $\mathcal{F}_{2}$, respectively, satisfy 
$V_{1}=V_{2}$ on 
$M_{1}\cap M_{2}\neq \emptyset $, then 
$\mathcal{F}_{1}=\mathcal{F}_{2}$. 
This would imply that given a VF $(V,M)$,  the mean
domain $M$ is the largest open interval on which $V$ is positive real
analytic.

\medskip\noindent
Suppose that we would denote the function
$\big(k_\mu^{\prime}\big)^{-1}(\cdot)$ by $\psi_\mu(\cdot)$.
Then we get by differentiating
\begin{equation*}
\psi_\mu^{\prime}(m) = \frac{1}%
{k_\mu^{\prime\prime}\big((k_\mu^{\prime})^{-1}(m)\big)}
=\frac{1}{V(m)}.
\end{equation*}
Similarly, when we would define
$\phi_\mu(\cdot)\doteq k_\mu\big((k_\mu^{\prime})^{-1}(\cdot)\big)\big)$,
we get by differentiating
\begin{equation*}
\phi_\mu^{\prime}(m)=\frac{m}{V(m)}.
\end{equation*}
Again, we find these derivatives to be not dependent on the specific
distribution, but only on the VF. 
Remarkebly, then their antiderivatives do not depend either on $\mu$.
Thus, rather defining the functions 
$\psi_\mu(\cdot)$ and $\phi_\mu(\cdot)$ as above for specific 
measure $\mu$, we define these as 
functions on the mean domain $M$ of NEF $\mathcal{F}$
as primitives of $1/V(m)$ and $m/V(m)$, respectively, i.e.,
\begin{equation}\label{c}
\psi(m)\doteq \int \frac{dm}{V(m)},  
\end{equation}
and
\begin{equation} \label{d}
\phi(m)\doteq \int \frac{mdm}{V(m)}.
\end{equation}
As a result, consider that a VF $(V, M)$ 
of an NEF $\mathcal{F}=\mathcal{F}(\mu)$ is given,
and suppose that we 
choose any two
primitives $\psi(m)$ and $\phi(m)$
of $1/V(m)$\ and $m/V(m)$, respectively.
Then there exists a positive Radon measure $\mu^*\in\mathcal{M}$ 
such that 
\begin{equation}\label{4}
\phi(m)=\log 
\int_{\mathbb{R}}\exp (\psi(m)x)\mu^*(dx),
\quad m\in M,  
\end{equation}
and
\begin{equation}\label{5}
\mathcal{F}=\mathcal{F}(\mu^*)=
\Big\{ F\big(m,\mu^*(dx)\big)=
\exp \big( x\psi(m)-\phi(m)\big) \mu^*(dx) : m\in M\Big\}.  
\end{equation}
The reparameterization of $\mathcal{F}$ in (\ref{5}) is called the mean
value parameterization of $\mathcal{F}$ 
\citep[see][Proposition 2.3]{letac1990}.
Accordingly, an NEF has two natural presentations: one is parameterized by
canonical parameter $\theta $ and is given in (\ref{3}) and the second by
the mean parameter $m$, and is given in (\ref{5}). However, as far as
statistical applications concern, the rather more important presentation is
the mean value parameterization (as $\theta $ is just an artificial
parameter - the argument of the corresponding Laplace transform).

\medskip\noindent
We now present the definitions of steep NEFs and EDMs:

\begin{itemize}
\item \textbf{Steep NEFs}: 
An NEF $\mathcal{F}(\mu)$ is called 
steep $\Leftrightarrow $ its cumulant transform 
$k_{\mu}(\theta )$ is essentially smooth convex function on 
$D_{\mu }$ (defined in (\ref{1})) 
$\Leftrightarrow M=\mathrm{int}\, C_{\mu }$ 
\citep[c.f.,][]{barndorff1978,letac1990}.
We shall refer to this definition in the
sequel.

\item \textbf{EDMs}: 
Let $\mathcal{F}=\mathcal{F}(\mu )$ be an NEF generated
by $\mu$ with  
Laplace and cumulant transforms $L_{\mu }$ and 
$k_{\mu }$, respectively. Denote
\begin{equation*}
\Lambda =\left\{ p\in \mathbb{R}^{+} : L_{\mu }^{p}
\text{ is a Laplace transform of some measure }
\mu_{p}\right\} ,
\end{equation*}
then $\Lambda $ is nonempty due to convolution, 
and it is called the Jorgensen
set (or the dispersion parameter space in the terminology of EDMs). 
It has been shown that $\Lambda =\mathbb{R}^{+}$ iff $\mu $ 
(and thus all members of $\mathcal{F}(\mu)$) is infinitely divisible. 
If $p\in \Lambda $, 
 the cumulant function of $\mu_p$ is
\[
k_{\mu_p}(\theta)=\log L_\mu^p(\theta) = pk_\mu(\theta).
\]

\noindent
Hence, the NEF
generated by $\mu_{p}$ is the set of probability distributions
\begin{equation}
\mathcal{F}_p=\mathcal{F}(\mu_p)\doteq
\Big\{ F\big(\theta ,\mu_{p}(dx)\big)
=\exp\big( \theta x-pk_{\mu}(\theta) \big)\mu_{p}(dx) :
\theta \in \Theta _{\mu_{p}}=\Theta_{\mu }\Big\}.  \label{Jor1}
\end{equation}
Furthermore, the mean parameterization goes similarly as above.
Denote the VF of $\mathcal{F}_p$ by $(V_p, M_p)$, 
and denote primitives of
$1/V_p(m)$ and $m/V_p(m)$ by $\psi_p(m)$ and $\phi_p(m)$,
respectively. Then, there is a positive Radon measure $\mu_p^*$
such that
\begin{equation*}
\mathcal{F}_p =\mathcal{F}(\mu_p^*)=
\Big\{ F\big(m ,\mu^*_{p}(dx)\big)
=\exp\big( x\psi_p(m)-\phi_p(m)\big)\mu_{p}^*(dx) :
m \in M_p\Big\}.
\end{equation*} 
Its VF $(V_p, M_p)$ satisfies
\begin{equation}\label{Jor2}
V_p(m)
=k_{\mu_p}^{\prime\prime}\big((k_{\mu_p}^{\prime})^{-1}(m)\big)
=pk_{\mu}^{\prime\prime}\big((k_{\mu}^{\prime})^{-1}(m/p)\big)
=pV(m/p),
\end{equation}
and $M_p=pM_\mu$.
Hence, the primitive $\psi_p(m)$ is obtained by
\begin{equation}\label{e:psip}
\psi_p(m)=\int\frac{dm}{V_p(m)}=
\int\frac{dm}{pV(m/p)}=\int\frac{d(m/p)}{V(m/p)}
=\psi(m/p).
\end{equation}
Similarly for the primitive $\phi_p(m)$ we get
\begin{equation}\label{e:phip}
\phi_p(m)=\int\frac{mdm}{V_p(m)}=
\int\frac{mdm}{pV(m/p)}=p\int\frac{(m/p)d(m/p)}{V(m/p)}
=p\phi(m/p).
\end{equation}
In this way, the mean parameterization of the NEF
$\mathcal{F}_p$ becomes
\begin{equation}\label{e:meanparmp}
\mathcal{F}_p =
\Big\{ F\big(m ,\mu^*_{p}(dx)\big)
=\exp\big( x\psi(m/p)-p\phi(m/p)\big)\mu_{p}^*(dx) :
m \in pM\Big\}.
\end{equation}

\medskip\noindent
The set of NEFs 
\begin{equation*}
\cup _{p\in \Lambda }\,\mathcal{F}_p
\end{equation*}%
was termed by \citet{jorgensen1987} the EDM corresponding to $\mu$. In
particular if $\Lambda =\mathbb{R}^{+}$ (i.e., $\mu $ is infinitely
divisible) then EDMs are used to describe the error component in generalized
linear models. 
\end{itemize}

\medskip\noindent
Many  types of VFs of NEFs have been presented and discussed in the
literature \citep[for a thorough survey see][]{barlev2017}.
Related to our study are VF's having a polynomial structure,
for which all of the respective cumulants and moments are
also polynomials. In quite generality, \citet{barlev1987}
and  \citet[Theorem 3.2]{letac1990} showed that
any $r$-th degree polynomial of the form
\begin{equation}
V(m) = \sum\limits_{i=1}^{r}a_{i}m^{i},
\quad m\in \mathbb{R}^{+},\, r\in \mathbb{N},
\label{b}
\end{equation}%
where $a_{i}\geq 0, i=1,\ldots,r$, and $\sum_{i=1}^{r}a_{i}>0$,
is a VF of an infinitely divisible NEF. 
Special cases are quadratic VFs \citep{morris1982},
the six strictly cubic VFs \citep{letac1990},
the Tweedie class having VFs of the
form $V(m)=\alpha m^\gamma$
\citep{tweedie1984,barlev1986,jorgensen1987,jorgensen1997},
the Hinde-Dem\'etrio class $V(m)=m+m^\gamma$
\citep{hinde1998,kokonendji2004b,kokonendji2007},
Poisson-Tweedie class with VF
$V(m)=m + \alpha m^\gamma$
\citep{kokonendji2004b,jorgensen2016}, and
Poisson-exponential-Tweedie models which have the VF
$V(m) = m+m^2+\alpha m^\gamma$ \citep{abid2020}.

\medskip\noindent
Recall that for a given VF $(V,M)$, 
$\psi(m)$ and $\phi(m)$ are primitives
of $1/V(m)$ and $m/V(m)$,  see
\eqref{c} and \eqref{d},  respectively.
Accordingly, if $V$ is of the general form (\ref{b}) it is not possible to
explicitly express $\psi(m)$ and $\phi(m)$, in which case the mean
value parameterization \eqref{e:meanparmp} 
is useless for any practical consideration.
If, however, for some special cases of the $a_{i}$'s coefficients, it can be
calculated nicely and explicitly then so can be the corresponding likelihood
function based on an appropriate random sample. This fact has a tremendous
significance in statistical inference. 

\bigskip\noindent
After this long introduction we arrive at the crux of the paper. 
\citet{letac1990} proved a proposition that
(a) characterizes the VF $(V,M)$ of NEFs that are generated 
by counting measures on $\mathbb{N}_0$, and (b)
gives an expression for the generating measures of such
NEFs. 
For our study, part (b) is relevant, and therefore we state it
below. 
 
\begin{proposition}[Part (b) of Proposition 4.4 in 
\citet{letac1990}] \label{p:prop4.4}
Let $\mathcal{F}=\mathcal{F}(\mu)$ be an NEF on $\mathbb{N}_0$
with VF $(V,M)$,  and with $\mu_{0}>0$ and $\mu_{1}>0$. 
Choose any two primitives $\psi(m)$ and $\phi(m)$
of $1/V(m)$ and $m/V(m)$, respectively, that satisfy

\begin{equation}
\lim_{m\to 0} \phi^{\prime}(m)=1,  \label{cond1}
\end{equation}
and
\begin{equation}
\lim_{m\to 0} m\exp\big(-\psi(m)\big) = 1.  \label{cond2}
\end{equation}
Define, 
\[
G(m) \doteq m\exp\big(-\psi(m)\big),\quad m\in M.
\]
Then the NEF $\mathcal{F}$ is generated by the
measure $\mu^*\in \mathcal{M}$ whose entries are computed by

\begin{equation}\label{LM1}
\begin{cases}
\mu^*_{0}=\exp\big(\phi(m)\big)\big|_{m=0},\\ 
\mu^*_{n}=\frac{1}{n!}
\big(\frac{d}{dm}\big)^{n-1}
\Big(\big(\exp (\phi(m))\big)\times
\phi^{\prime}(m)\times \big(G(m)\big)^{n}\Big)\Big|_{m=0},
\quad n=1,2,\ldots.
\end{cases}
\end{equation}
\end{proposition}

\noindent
Condition \eqref{cond1} is necessary for the NEF to be concentrated
on the nonnegative integers, however, it leaves infinitely many
choices for the function $\phi$ which are all equal upto an additive
constant. For our purposes it is most convenient to impose 
the condition
\begin{equation}\label{con3}
\lim_{m\to 0}\phi(m) = 0.
\end{equation}
This condition gives $\mu^*_0=1$, which is convenient in our analysis
of the zero-inflation properties that we discuss in the sequel.
Secondly, we remark that condition \eqref{cond2}
could be relaxed. We state this as a lemma.

\begin{lemma}\label{l:prop4.4}
Proposition \ref{p:prop4.4} holds for \eqref{cond2}
relaxed to
\[
\lim_{m\to 0} m\exp\big(-\psi(m)\big) = c,
\]
where $c>0$ is any positive constant.
\end{lemma}
\begin{proof}
Denote the $\psi$-primitive that satisfies this condition by $\psi_c(m)$.
Clearly, $\psi_c(m) = \psi_1(m) - \log c$, where $\psi_1(m)$
is the primitive of Proposition \ref{p:prop4.4}.
When we implement the computations \eqref{LM1} with 
$\psi_c$ in stead of $\psi_1$, the resulting measure $\widetilde{\mu}$
is an exponential shift of the $\mu^*$ of \eqref{LM1}.
\end{proof}

\noindent
Returning to the concept of EDM, it is now obvious that
the generating measures $\mu_p^*, p\in\Lambda$
are computed similarly as in Proposition \ref{p:prop4.4}.

\begin{corollary}\label{c:prop4.4}
Let $\{\mathcal{F}_p, p\in\Lambda\}$ be the EDM originating
from the NEF $\mathcal{F}=\mathcal{F}(\mu)$ of Proposition \ref{p:prop4.4}. 
For any $p\in\Lambda$,
a generating measure $\mu_p^*$ is obtained by
the computations \eqref{LM1} in which the primitives
are $\psi_p(m)=\psi(m/p)$, and $\phi_p(m)=p\phi(m/p)$.
\end{corollary}
\begin{proof}
Because (i) the mean domain of the NEF $\mathcal{F}_p$ is $M_p=pM$,
and (ii) $\phi_p^{\prime}(0)=\phi^{\prime}(0)=1$,
the NEF $\mathcal{F}_p$ is concentrated on $\mathbb{N}_0$.
Furthermore, 
\[
\lim_{m\to 0} m\exp\big(-\psi_p(m)\big)
=p \lim_{m\to 0} (m/p)\exp\big(-\psi(m/p)\big) = p.
\]
Now apply Lemma \ref{l:prop4.4}.
\end{proof}

\medskip\noindent
Two subclasses of polynomial variance functions (\ref{b}), which satisfy
the conditions that the corresponding NEFs are concentrated on 
the nonnegative integers, are presented in \citet{letac1990}.
These are
\begin{equation}\label{LM}
V(m)=m\prod_{i=1}^{r}\Big( 1+\frac{m}{p_{i}}\Big),
\quad M=\mathbb{R}^{+},
\end{equation}
and
\begin{equation}\label{LM2}
V(m)=\frac{m}{\prod_{i=1}^{r}\left( 1-\frac{m}{p_{i}}\right) }, \quad
M=\big(0,\min (p_{1},\ldots,p_{r})\big),
\end{equation}
where
$p_{i}>0,i=1,2,\ldots$, and  $r\in \mathbb{N}$.

\medskip\noindent 
A few simple cases ($r\leq 2$) of these subclasses have been considered
in the statistical literature, for instance, the VF
\[
V(m) = m\Big( 1+\frac{m}{p}\Big)^2,
\]
which results in the the Abel distribution
\citep[also known as generalized Poisson; c.f.][]{consul1989,consul2006}.

\medskip\noindent
For $r\geq 3$, the corresponding $\mu _{n}$'s 
in \eqref{LM1}, and
thus also the NEF probabilities in (\ref{5}), cannot be presented neither in
closed and explicit forms nor in terms of infinite sum (or some
transcendental functions). They can be derived only through numerical
calculations by either mathematical software as Mathematica,
or Maple
or by writing appropriate computer programs 
in 4th generation languages as Matlab, R and Python.
This explains our statement
above that many NEFs (at least with polynomial VF structure and degree 
$r\geq 3$) have not been used for statistical modeling or applications for
the mere fact that they have not been known before and thus not been
considered and investigated. Therefore in this paper we intend to correct to
a certain extent the `injustice' caused to these discrete NEFs. Notice
however an important point. When we refer to (\ref{5}) in a Bayesian
framework and when (\ref{5}) serves as a prior distribution then the 
$\mu_{n}$'s calculation becomes superfluous and redundant when calculating the
posterior distribution, as one can choose arbitrarily any two primitives 
$\psi(m)$ and $\phi(m)$. We shall further relate to this point in
Section 3.

In particular, we present in the sequel  
two  subclasses: one of the form
(\ref{LM}), and one of the form (\ref{LM2}). 
For convenience we say
classes although they are subsets of (\ref{LM}) and (\ref{LM2}). 
These classes of VFs were chosen because of the relative simplicity 
of the calculations of $\psi(m)$ and $\phi(m)$ for which explicit expressions
are available. The two VF classes are
\begin{align}
V_p(m) &= m\Big(1+\frac{m}{p}\Big)^{r}, \quad M=\mathbb{R}^{+}; \label{ABM1}\\
V_p(m) &= \frac{m}{\big(1-\frac{m}{p}\big)^{r}}, 
\quad M=(0,p),\label{LMclass}
\end{align}
where $p>0$ and $r\in \mathbb{N}_{0}$.
Later we will
coin each class a name and discuss its properties. However, at this point,
we will notice a very important fact. Both classes are of the form 
(\ref{Jor2}) representing VFs of EDMs. Consequently, their corresponding
probabilities belong to the realm of EDMs.

\medskip\noindent
The paper is organized as follows. In Section \ref{s:further} 
we will discuss further
important aspects related to the practical implementation of 
Proposition \ref{p:prop4.4}.
In Section \ref{s:analysis} 
we elaborate the two 
classes presented in 
\eqref{ABM1} and \eqref{LMclass}. 
For each class we derive
expressions for $\psi(m)$ and $\phi(m)$ which fulfills the premises of
Proposition \ref{p:prop4.4}. We then describe some of their properties. In
particular it will be shown that the corresponding NEFs' distributions are
overdispersed and zero-inflated in ascending order in $r$.
A numerical example of real data,
presented in Section \ref{s:num}, 
compares the performances
of our two classes 
to other well used discrete distributions. This
example demonstrates 
the superiority of the members of 
these classes for larger power $r$ of the polynomial VF,
vis-a-vis all other distributions. 
Section \ref{s:conclusion} is devoted to some
concluding remarks.


\section{Further Aspects and Analysis
and Presentation of the Two Classes}\label{s:further}

As stated above, our goal is to locate classes of VFs, 
subclasses of (\ref{LM}) and (\ref{LM2}), for which we can derive 
explicitly and relatively simple
expressions both for the $\psi(m)$ and $\phi(m)$ functions.
Our first class has variance function
(see \eqref{ABM1})
\[
V_p(m) = m\Big(1+\frac{m}{p}\Big)^{r},
\]
which belongs to the realm of (\ref{LM}). 
The special cases $r=0$, $r=1$ and $r=2$ correspond, respectively, to the
Poisson, negative binomial and Abel (or generalized Poisson) NEF's. 
The class in (\ref{ABM1}) is called the ABM class, as it was first presented by
\citet{awad2016} in a Bayesian framework. Further details
regarding such a Bayesian framework for the ABM class can be found in
\citet{barlev2017}. The second class does not have a
polynomial structure. Its variance function has the form 
(see \eqref{LMclass}
\[
V_p(m) = \frac{m}{\big(1-\frac{m}{p}\big)^{r}}.
\]
We call this class the LM class, as being a subclass 
of \eqref{LM2} which was presented in \citet{letac1990}.

Before we proceed to discuss the two classes separately in the
subsections below, we will present a number of general comments regarding
these classes (as well as any other classes too).

\begin{enumerate}
\item \textbf{Steepness}: 
The NEFs corresponding to the two classes of VFs
are concentrated on $\mathbb{N}_{0}$, thus their convex support is 
$C=[0,\infty )$. Hence, the first two classes (ABM and LMS) belong to steep
NEFs as their mean domain $M=\mathbb{R}^{+}$ coincides with $intC$. In
contrast, the LMNS class is nonsteep as the corresponding mean domain $M=%
\mathbb{(}0,p)$ is a proper subset of $(0,\infty )$.

\item \textbf{Infinitely divisibility and EDMs}: All of the three classes
constitute infinitely divisible NEFs as they are subsets of (\ref{b}) and
thus the dispersion parameter space $\Lambda =\mathbb{R}^{+}$\ (i.e., they
are VFs for all $p\in \mathbb{R}^{+}$). Thus, as indicated above, they
establish EDMs.

\item \textbf{The form of }$\Theta$: 
We notice that the set $\Theta $ is the image
of $\mathbb{R}^{+}$ for the ABM clas, and the image of $(0,p)$ of
the LM class by the map $m\mapsto \theta =\psi (m)$. Thus, it has the
form $(-\infty ,q)$, for some $q\in \mathbb{R}$. Obviously, the calculation
of the inverse function $m\mapsto \theta =\psi (m)$ cannot be done in an
elementary way for $r>2$ (and sometimes also not for $r=2$). 

If $\mu $ is bounded then one can impose conditions on $\mu $
to be a probability. The question arises, therefore, when $\mu $ is bounded.
The following simple lemma (whose proof is presented, without any loss of
generality, for the ABM class only) provides an answer.
\end{enumerate}

\begin{lemma}\label{l:bounded}
The generating $\mu $ of the NEF $\mathcal{F}$ is bounded iff $q\geq 0.$
\end{lemma}
\begin{proof}
$\Longleftarrow :$ If $q>0$ then $k(0)<\infty $. Recall that $e^{k(0)}$ is
the total mass of $\mu$. If $q=0$, then for $\theta <0$ we assume that $%
\lim_{\theta \rightarrow -\infty }k(\theta )=0$ and write
\begin{eqnarray*}
k(\theta ) &=&\int_{-\infty }^{\theta }k^{\prime }(t)dt=\int_{0}^{k^{\prime
}(\theta )}k^{\prime }(\psi (s))\psi ^{\prime }(s)ds=\int_{0}^{k^{\prime
}(\theta )}\frac{s}{V(s)}ds \\
&=&\int_{0}^{k^{\prime }(\theta )}\frac{ds}{(1+\frac{s}{p})^{r}},
\end{eqnarray*}
where in the last equality we used VF corresponding to the ABM class. Since 
$\lim_{\theta \rightarrow 0}k^{\prime }(\theta )=\infty $ we can claim 
that
\begin{equation*}
\lim_{\theta \rightarrow 0}k(\theta )=\int_{0}^{\infty }
\frac{ds}{(1+\frac{s}{p})^{r}}=\frac{p}{r-1},r\geq 2.
\end{equation*}
This shows that when $q=0$, the total mass of $\mu $ is $e^{p/(r-1)}$. If 
$\mu $ is normalized to make it a probability then 
$\lim_{\theta \rightarrow -\infty }k(\theta )=0$ 
is no longer fulfilled after such a normalization.

\medskip\noindent
$\Longrightarrow :$ If $q<0$ the measure $\mu $ is unbounded since $0$ does
not belong to the closure of $\Theta $.
\end{proof}

\begin{enumerate}
\setcounter{enumi}{3}
\item \textbf{Cumulants and moments}: As we have already mentioned, the
cumulants (and thus also moments and central moments) of the ABM class
will also be polynomials. Their calculations are based on the
following simple result \citep[c.f.,][]{barlev1992b}.
Define an operator $L$ acting on $V$ by 
$L(V) =VV^{\prime}$
and define $L_{j}(V)=L\big(L_{j-1}(V)\big)$ for
$j=1,2,\ldots$, with $L_{0}(V)=V$.
Consider an NEF $\mathcal{F}=\mathcal{F}(\mu)$
with VF $(V,M)$, and let $\psi, \phi, \mu^*$
resulting from Proposition \ref{p:prop4.4}.
The $j$-th  cumulant of $\mu^*$  
expressed in terms of $m$, is given by
\begin{equation}
k^{(j)}_{\mu^*}(m)\doteq k^{(j)}_{\mu^*}\big(\psi(m)\big)
=L_{j-2}\big(V(m)),\text{ for all }
j=2,3,\ldots,\text{ and }m\in M.  \label{cumulants}
\end{equation} 
Let us find in this way the first four cumulants,
and moments in the ABM class.
Fix $m,p,r$, thus $V_p(m)=m(1 + m/p)^r$. Then,
\begin{align*}
k^{(1)}_{\mu^*_p}(m) &= m\\ 
k^{(2)}_{\mu^*_p}(m) &= V_p(m)=m\big(1+\frac{m}{p}\big)^r\\ 
k^{(3)}_{\mu^*_p}(m) &= L_1\big(V_p(m)\big)=V_p(m)V_p^\prime(m) = 
m\big(1+\frac{m}{p}\big)^{2r-1}
\big(1 + \frac{m}{p}(1+r)\big)\\ 
k^{(4)}_{\mu^*_p}(m) &= L_2\big(V_p(m)\big)
=L\Big(L_1\big(V_p(m)\big)\Big)\\
&=V_p(m)\Big(L_1\big(V_p(m)\big)\Big)'=
V_p(m)\Big(V_p(m)V_p(m)'\Big)'\\
&=V_p(m)\Big(\big(V_p(m)\big)'\Big)^2 + \big(V_p(m)\big)^2V_p''(m)\\
&=m\big(1+\frac{m}{p}\big)^{3r-2}
\Big(1 + \frac{m}{p}(1+2r)\big(2+\frac{m}{p}(1+3)\big)\Big).
\end{align*}
Moreover, let $X_p$ be the random variable associated with
parameters $m,p,r$ in the EDM of the ABM class, thus
\begin{equation}
\mathbb{P}(X_p=n)= (\mu^*_p)_{n}\exp\big(\psi_p(m)n-\phi_p(m)\big) ,
\quad
n=0,1,\ldots.  \label{probability of X}
\end{equation}
Clearly, it has mean $m$ and variance $V_p(m)$.
To derive its higher central moments it is most convenient to 
consider its central moment generating function,
\[
C_{X_p}(t)\doteq \mathbb{E}\big[\exp\big(t(X_p-m)\big)\big]
=\exp\big(k_{X_p}(t) -mt\big),
\]
where $k_{X_p}(t)$ is the cumulant function of $X_p$,
\[
k_{X_p}(t)\doteq \log \mathbb{E}\big[\exp(tX_p)\big]
=k_{\mu^*_p}\big(\psi_p(m)+t\big) - \phi_p(m).
\]
Thus,
\[
C_{X_p}(t) = \exp\Big(k_{\mu^*_p}\big(\psi_p(m)+t\big) - mt - \phi_p(m)
\Big),
\]
and the $j$-th central moment is obtained by $C_{X_p}^{(j)}(0)$.
From these central moments we get the skewness and kurtosis.
After doing the calculus,
\begin{align*}
\text{skewness} &=\frac{C_{X_p}^{(3)}(0)}{\big(C_{X_p}^{(2)}(0)\big)^{3/2}}
=\frac{k^{(3)}_{\mu^*}(m)}{\big(k^{(2)}_{\mu^*}(m)\big)^{3/2}}
=\frac{V_p(m)V_p^\prime(m)}{V_p(m)\sqrt{V_p(m)}}\\
&= m^{-1/2}\,\big(1+\frac{m}{p}\big)^{(r/2)-1}\,
\big(1 + \frac{m}{p}(1+r)\big)\\
\text{kurtosis} &= \frac{C_{X_p}^{(4)}(0)}{\big(C_{X_p}^{(2)}(0)\big)^2}
= \frac{k^{(4)}_{\mu^*}(m)}{\big(k^{(2)}_{\mu^*}(m)\big)^{2}}
=\frac{V_p(m)\big(V_p^\prime(m)\big)^2+V_p^2(m)V_p^{\prime\prime}(m)}{V_p^2(m)}\\
&=\frac{\big(V_p^\prime(m)\big)^2}{V_p^2(m)} + V_p^{\prime\prime}(m)
=m^{-1}\big(1+\frac{m}{p}\big)^{r-2}
\Big(1 + \frac{m}{p}(1+2r)\big(2+\frac{m}{p}(1+3)\big)\Big).
\end{align*}

\item 
\textbf{Overdispersion}: Recall that in statistics, overdispersion is
the presence of greater variability in a data set than would be expected
based on a given statistical model. For instance, the Poisson NEF which is
commonly used in practice to model count data (e.g., number of insurance
claims; number of customers arriving into a queueing system). The
theoretical mean and variance for the Poisson model are equal. On the other
hand, in a large number of empirical data sets, the sample variance is
considerably larger than the sample mean. Consequently, researchers have
tried to model such data sets by families of distributions, such as the
negative binomial and the generalized Poisson -Abel) distributions, for
which the variance is larger than the mean. The statistical literature is
full of articles on this subject, but we refrain from citing them for the
sake of brevity. 

Consider the polynomial VF in (\ref{LM}) and denote explicitly
its degree; i.e.,
\begin{equation*}
V_{r+1}(m) \doteq
m\prod_{i=1}^{r}\big( 1+\frac{m}{p_{i}}\big).
\end{equation*}
Then
trivially we have that the larger 
the degree of the polynomial,
the larger is $V_{r}$, i.e.,
\begin{equation*}
m = V_{1}(m) < V_{2}(m) < \cdots.
\end{equation*}
Firstly, the latter property indicates that all of the associated NEFs 
distributions are overdispersed with respect to the Poisson
distribution, and secondly, there is an ascending order in $r$ of such an
overdispersion. Similarly, this overdispersion property trivially holds also
for the second class  
given in (\ref{LM2}). As the ABM, and the
LM are subclasses of (\ref{LM2}) or(\ref{LM}), they share the same
overdispersion property. Moreover, one can simply realize that for any 
degree $r$ one has $V_{r}^{ABM}<V_{r}^{LM}$, i.e., 
the LM class is more overdispersed than the ABM one.
\end{enumerate}

\begin{remark}\label{r:HD}
Consider a simple polynomial VF the form 
\begin{equation}\label{HD1}
V(m)=m(1+m^{r}), \quad M=\mathbb{R}^{+},
\end{equation}
where  $r\in \mathbb{N}_{0}$. This structure of polynomial
VFs is called the HD class, as it has been introduced 
by \citet{hinde1998} for overdispersed
models and characterized by \citet{kokonendji2007} to analyze
overdispersed and zero-inflated count data. Further theoretical and data
analysis of the HD class can be found in \citet{kokonendji2004b,kokonendji2008}.
Though the HD
class is not a special case of (\ref{LM}), this class of VFs fulfill the
premises of \citet[Proposition 4.4]{letac1990}. Thus, the corresponding NEFs
distributions are supported on $\mathbb{N}_{0}$ and 
the respective $\mu^*_{n}$'s can be computed by Proposition \ref{p:prop4.4}. 
Nonetheless, for the HD class, although $\psi(m)$ in (\ref{d}) 
can be nicely expressed as 
\begin{equation*}
\psi (m)=\int \frac{dm}{m(1+m^{r})}=
\log m-\frac{1}{r}\log (m^{r}+1) +a,
\end{equation*}
$\phi(m)$, however, has a very complex expression of the form
\begin{equation*}
\phi(m)=\int \frac{dm}{1+m^{r}}
=m_{2} F_{1}(1,\frac{1}{r};1+\frac{1}{r};-m^{r})+b,  
\end{equation*}
where
\begin{equation*}
_{2}F_{1}(a,b;c;z)=\sum_{k=0}^{\infty }
\frac{(a)_{k}(b)_{k}}{(c)_{k}}\frac{z^{k}}{k!}
\end{equation*}
is the generalized hypergeometric function of type $2$ and $1$,
respectively, and $(d)_{k}$ is the Pochhammer symbol. Such an expression
for $\phi(m)$ makes the corresponding mean value parameterization of
the HD class unuseful for practical considerations.
\end{remark}


\section{Some Analysis of the ABM and LM Classes}\label{s:analysis}

In the following three subsections we will discuss the two classes in two
aspects. One is to find explicit expressions for $\psi(m)$ and 
$\phi(m)$ functions that satisfy the conditions
\eqref{cond1}, \eqref{cond2}, and \eqref{con3},  and for which 
Proposition \ref{p:prop4.4} is applicable. The second aspect is to show
that the distributions of the relevant NEFs are zero-inflated with respect
to the Poisson NEF and among themselves in an ascending order. Recall that a
zero-inflated model is a statistical model based on a zero-inflated
probability distribution, i.e. a distribution that allows for frequent
zero-valued observations. In various insurance data the probability of the
event of no claims during the insured period is rather large and the Poisson
model does not fit. Various other models have been suggested in the realm of
zero-inflated models in which the probability of zero is larger than the
probability of nonzero. Such zero-inflated distributions are naturally
overdispersed relative to the Poisson distribution. On this subject, too,
the statistical literature is full of relevant articles, but we refrain from
quoting them for reasons of brevity.

In each subsection we provide two propositions. One relates to the
computations of the $\psi(m)$, $\phi(m)$, and $G(m)$ functions
fulfilling the conditions \eqref{cond1}, \eqref{cond2}, and \eqref{con3}; 
the second proposition relates to the zero-inflated
property.


\subsection{The ABM Class}\label{ss:abm}
The ABM class has been first introduced by \citet{awad2016}
for implementing mortality projections in actuarial science. In this
respect, the Lee-Carter model \citep{lee1992}, and variants thereof
\citep[e.g.,][]{renshaw2006} is a largely acceptable method of
mortality forecasting. \citet{awad2016} have dealt with
predicting mortality rates by embedding the Lee-Carter model within a
Bayesian framework. They used the ABM class of counting distributions as
alternatives to the Poisson counts of events (deaths) under the Lee-Carter
modeling for mortality forecast and showed that members of the ABM class
predicts better than the Poisson the mortality rates of elderly age people.
This has been demonstrated for national data of the US, Ireland and Ukraine.
Since the Bayesian approach was involved, it was not relevant there to
calculate neither the constants of integration for the primitives $\psi$
and $\phi$,
nor the $\mu _{n}^*$'s
in (\ref{LM1}), as these constants and mass points are cancelled out while
computing the appropriate posterior distribution 
\citep[for further details see][]{barlev2017}. 
They also did not demonstrate how the
general expressions are obtained for $\psi $ and $\phi$. Therefore, we
will provide the appropriate proof.

\begin{proposition}\label{p:abmfunctions}
Consider  the ABM class with the variance function 
given in (\ref{ABM1}) for $r\geq 1$.
The corresponding 
$\psi_p(m) $, $\phi_p(m)$ and $G_p(m)$ functions
fulfilling the conditions \eqref{cond1}, \eqref{cond2}, and \eqref{con3}, have
the forms
\begin{align}
\psi_p(m)&=\log \frac{m}{m+p} + \sum_{j=1}^{r-1}
\frac{1}{j}\Big(\frac{p^{j}}{(m+p)^{j}}-1\Big); \label{psi ABM1}
\\
\phi_p(m)&=
\begin{cases}
p\log\frac{m+p}{p}, & \quad (r=1);\\
\frac{p}{r-1}\Big(1-\big(\frac{p}{m+p}\big)^{r-1}\Big), &\quad (r\geq 2);
\end{cases}
\label{psi1 ABM1} \\
G_p(m)&= (m+p)\exp \Big( \sum_{j=1}^{r-1}
\frac{1}{j}\big( \frac{p^{j}}{(m+p)^{j}}-1\big) \Big).  \label{G for ABM1}
\end{align}
\end{proposition}

\medskip\noindent
We exclude the trivial case $r=0$ (Poisson). As usual, an empty sum
$\sum_{j=1}^{0}\cdot=0$ in \eqref{psi ABM1} and \eqref{G for ABM1}
in the $r=1$ case.

\begin{proof}
Using the EDM properties \eqref{e:psip} and \eqref{e:phip},
it suffices to determine the $\psi(m), \phi(m)$ and $G(m)$
functions of the originating ABM NEF with variance function
$V(m)=m(1+m)^r$. Consider
\[
\int \frac{1}{V(m)}\,dm = \int \frac{1}{m(1+m)^r}\,dm, 
\]
and apply polynomial division,

\begin{align*}
&\frac{1}{m(1+m)^{r}}-\frac1m = 
-\frac1m\Big(1 - \big(\frac{1}{1+m}\big)^r\Big)\\
&= -\frac1m\Big(1 - \frac{1}{1+m}\Big)\Big(1 + \frac{1}{1+m}
+\big(\frac{1}{1+m}\big)^2+\cdots+\big(\frac{1}{1+m}\big)^{r-1}\Big)\\
&=-\frac{1}{1+m}\sum_{j=0}^{r-1}\big(\frac{1}{1+m}\big)^{j}
=-\sum_{j=1}^{r}\frac{1}{(1+m)^{j}}.
\end{align*}%
Hence
\begin{align*}
& \psi(m) = \int \frac{1}{m(1+m)^r}\,dm
=\int \Big(\frac1m -\sum_{j=1}^{r}\frac{1}{(1+m)^{j}}\Big)\,dm\\
&=\log m - \log(1+m) + \sum_{j=2}^r\frac{1}{j-1}\,\frac{1}{(1+m)^{j-1}} +c
=\log \frac{m}{1+m} + \sum_{j=1}^{r-1}\frac{1}{j}\,\frac{1}{(1+m)^{j}} +c,
\end{align*}
where $c$ is the integration constant.
By the EDM property \eqref{e:psip} we get
\begin{equation}\label{psi 2 for ABM}
\psi_p(m)=\psi(m/p)=
\log \frac{m}{m+p}+\sum_{j=1}^{r-1}\frac{1}{j}\,\frac{p^{j}}{(m+p)^{j}}+c.
\end{equation}
Substituting $G_p(m)=m\exp\big(-\psi_p(m)\big))$, and solving
$G_p(0)=p$ (see Corollary \ref{c:prop4.4}), gives
\begin{align*}
& G_p(0)=\lim_{m\to 0}
(m+p)\exp \Big(-\sum_{j=1}^{r-1}\frac{1}{j}\,\frac{p^{j}}{(m+p)^{j}}
-c\Big)=p \, \exp\Big(-\sum_{j=1}^{r-1}\frac{1}{j}-c\Big)=p\\
& \Leftrightarrow \quad
c = -\sum_{j=1}^{r-1}\frac{1}{j},
\end{align*}
resulting in the forms (\ref{psi ABM1}) and 
(\ref{G for ABM1}). 
Recall that $\psi_p$ maps the mean domain $M$ into the space $\Theta$
of the natural parameter $\theta$.
From (\ref{psi ABM1}) we note that it is not possible to
express the inverse map ($m$ as a function of $\theta$),
implying that the corresponding
Laplace transform cannot be explicitly expressed as a function 
of $\theta $. This is the situation that will prevail in 
the other classes of VFs under
consideration.

\medskip\noindent
Now, to find $\phi_p(m)$ 
fulfilling (\ref{con3}), we again solve first for 
\[
\phi(m)=\int\frac{m}{V(m)}\,dm=
\int\frac{1}{(1+m)^r}\,dm = 
\begin{cases}
\log(1+m) +d, &\quad (r=1);\\
-\frac{1}{r-1}\,\frac{1}{(1+m)^{r-1}} + d, &\quad (r\geq 2),
\end{cases}
\]
with integration constant $d$. 
Applying EDM property $\phi_p(m)=p\phi(m/p)$, and the
condition $\phi_p(0)=0$,
we obtain
\[
d = 
\begin{cases}
0, &\quad (r=1);\\
\frac{p}{r-1}, &\quad (r\geq 2). 
\end{cases}
\]
Consequently, the primitive $\phi_p$ to
be used in Proposition \ref{p:prop4.4} has the form (\ref{psi1 ABM1}).
\end{proof}

\bigskip\noindent
Now we go back to discussing the ABM class in the context of zero-inflated
distributions. The probability mass at a point $n, n=0,1,\ldots$, 
is given in (\ref{probability of X}) with $\psi_p(m)$ and 
$\psi_p(m)$ given in 
(\ref{psi ABM1}) and (\ref{psi1 ABM1}), respectively. 
Specifically, using $(\mu^*_p)_0=1$,
the probability mass at $0$ is
\begin{equation}\label{probabily at 0 ABM}
P_{r}(0;p,m)\doteq (\mu^*_p)_{0}
\exp \big( -\phi_{p,r}(m)\big) 
=\exp \big( -\phi_{p,r}(m)\big),
\end{equation}
$r=1,2,\ldots$, 
where $\phi_{p,r}(m)$ denotes the $\phi_p(m)$ function 
when the VF has degree $r+1$.
Note that the probability at $0$ of the Poisson NEF with $r=0$,
is $e^{-m}$. We present the following proposition according to which
the probability at $0$ is an  
increasing function in $r$.
the associated distributions  become more and more
zero-inflated, a feature that enables the ABM class to serve as statistical
model for zero-inflated data.

\begin{proposition}\label{p:abmzero}
The zero-mass probability $P_{r}(0;p,m)$
in the ABM class is increasing in $r\geq 0$.
\end{proposition}

\begin{proof}
The zero probabiities are by using \eqref{probabily at 0 ABM} and \eqref{e:phip},
for any $r,p,m>0$,
\[
P_{r}(0;p,m) = \exp \big( -\phi_{p,r}(m)\big)
=\exp \big( -p\phi_{1,r}(m/p)\big).
\]
Thus, for showing that $P_{r}(0;p,m)$ is increasing, it suffices
to take $p=1$; i.e., to prove that $\phi_{1,r}(m)$ is decreasing, where
$\phi_{1,r}(m)$ is given in \eqref{psi1 ABM1} with $p=1$.
First, we consider $r>1$ for which
\[
\phi_{1,r}(m) = \frac{1}{r-1}\Big(1-\big(\frac{1}{1+m}\big)^{r-1}\Big).
\]
Define functions $\{f_x(s) : (0,\infty)\to\mathbb{R}, x\in(0,1)\}$ by
\[
f_x(s) = \frac1s(1-x^s).
\]
We shall argue that $f_x(s)$ is decreasing (in $s>0$)
for any $x\in (0,1)$. The derivative
\[
\frac{d}{ds} f_x(s) = -\frac{1}{s^2}(1-x^s) - \frac1s\,x^s \log x
=-\frac1s\Big(\frac1s + \frac1s\big(x^s + sx^s\log x\big)\Big).
\]
A simple calculus shows that $x^s + sx^s\log x>-1$ for $s>0$ for any $x\in(0,1)$
(for instance by determining its minimum).
Thus, $f_x(s)$ is decreasing (as function of $s>0$), and consequently,
$\phi_{1,r}(m)$ is decreasing for $r>1$, and  
$P_{r}(0;p,m)$ is increasing for $r\geq 2$ for any values of $p>0$ and $m>0$.

\medskip\noindent
To complete the proof we show (i) $P_{0}(0;p,m)<P_{1}(0;p,m)$,
and (ii) $P_{1}(0;p,m)<P_{2}(0;p,m)$.

\begin{enumerate}[(i).]
\item
$P_{0}(0;p,m)=e^{-m}$ is the zero-probability of the Poisson distribution
with mean $m$. From \eqref{psi1 ABM1} we see 
$P_1(0;p,m)=\big(p/(m+p)\big)^p$. Thus
\begin{align*}
& P_{0}(0;p,m)<P_{1}(0;p,m) 
\;\;\Leftrightarrow\;\;
e^{-m} < \big(\frac{p}{m+p}\big)^p
\;\;\Leftrightarrow\;\;
e^{-m} \big(\frac{m+p}{p}\big)^p < 1.
\end{align*}
The latter is an easy calculus to show for any $p>0$ and $m>0$,
for instance because the lefthand side equals 1 for $m=0$, but
is strictly decreasing for $m\geq 0$.

\item
It suffices to consider $p=1$ to conclude
\begin{align*}
& P_{1}(0;p,m)<P_{2}(0;p,m) 
\;\;\Leftrightarrow\;\;
P_{1}(0;1,m)<P_{2}(0;1,m)
\;\;\Leftrightarrow\;\;
\phi_{1,1}(m)>\phi_{1,2}(m)\\
& \Leftrightarrow\;\;
\log(m+1) > 1 - \big(\frac{1}{m+1}\big)^2.
\end{align*}
Equivalently, $\log x > 1 - 1/x^2$ for $x>1$, which is again a
simple calculus exercise.
\end{enumerate}

\end{proof}


\subsection{The LM Class}\label{ss:LM}

The LM class is given by VFs of the form (\ref{LMclass}). Recall that
the corresponding class of NEFs when $r\geq 1$ is non-steep with mean domain 
$(0,p)$, support $\mathbb{N}_{0}$ and convex support $[0,\infty )$. 
When $r=0 $ the corresponding VF is the Poisson one. 
\citet{bryc2005}
considered a special case $V(m)=m/(1-\frac{m}{p})$ on the mean domain $(0,p)$
and compute explicitly a measure $\mu $ such that 
$\mathcal{F}=\mathcal{F}(\mu )$ is an NEF supported on $\mathbb{N}_{0}$.

We will compute the primitives $\psi_p$ and $\phi_p$ that fulfill the 
conditions \eqref{cond1}, \eqref{cond2}, and \eqref{con3},
in much the same way as we did in Section \ref{ss:abm} for the ABM class.
In Appendix \ref{a:lmmisc} we will explore
a second way to find these functions, namely 
by following more closely the proof of Proposition 4.4 of \citet{letac1990},
which is based on using the Lagrange formula, and
express the $\mu^*_{n}$'s by means of Hermite polynomials. 

\begin{proposition}\label{p:lmfunctions}
Consider the LM class with the variance function  
given in (\ref{LMclass}) for $r\geq 0$.
The corresponding $\psi_p(m)$, $\phi_p(m)$, and $G_p(m)$ functions 
fulfilling the conditions \eqref{cond1}, \eqref{cond2}, and \eqref{con3}, have the
forms 
\begin{align}
\psi_p(m)&=\log \frac{m}{p} + \sum_{i=1}^{r}(-1)^i\frac{1}{i}\binom{r}{i}
\big(\frac{m}{p}\big)^{i}; 
\label{psi LMNS} \\
\phi_p(m)&=\frac{p}{r+1}\Big(1-\big(1-\frac{m}{p}\big)^{r+1}\Big);
\label{psi 1 LMNS}\\
G_p(m) &= 
p\,\exp\Big(
-\sum_{i=1}^{r}(-1)^i\frac{1}{i}\binom{r}{i}\big(\frac{m}{p}\big)^{i}\Big)
\label{G LMNS}
\end{align}
\end{proposition}

\begin{proof}
We determine the $\psi(m), \phi(m)$ and $G(m)$
functions of the originating LM NEF with variance function
$V(m)=m / (1-m)^r$. Namely,
\begin{align*}
&\int \frac{1}{V(m)}\,dm = \int \frac{(1-m)^r}{m}\,dm
=\int \frac1m\sum_{i=0}^r\binom{r}{i}(-m)^i\,dm\\
&= \int \frac1m\,dm + \sum_{i=1}^r(-1)^i \binom{r}{i}\int m^{i-1}\,dm
= \log m + \sum_{i=1}^r (-1)^i\frac1i\binom{r}{i}m^i + c,
\end{align*}
which gives by \eqref{e:psip}

\begin{equation}\label{psi of LMNS}
\psi_p(m)=\log \frac{m}{p} 
+\sum_{i=1}^{r}(-1)^i\frac{1}{i}\binom{r}{i}\big(\frac{m}{p}\big)^{i}+c.
\end{equation}
Substituting $G_p(m)=m\, \exp\big(-\psi(m)\big)$,
and solving $G_p(0)=p$ (see Corollary \ref{c:prop4.4}), gives
\[
G_p(0) = \lim_{m\to 0}p\,\exp\Big(
-\sum_{i=1}^{r}(-1)^i\frac{1}{i}\binom{r}{i}\big(\frac{m}{p}\big)^{i}-c\Big)
= p 
\quad \Leftrightarrow \quad
c = 0,
\]
resulting in the forms (\ref{psi LMNS}) and 
(\ref{G LMNS}).

\medskip\noindent
For $\phi(m)$,
\[
\int\frac{m}{V(m)}\,dm=
\int(1-m)^r\,dm =  -\frac{1}{r+1}(1-m)^{r+1} + d.
\]
Applying EDM property $\phi_p(m)=p\phi(m/p)$, and the
condition $\phi_p(0)=0$,
we obtain
\[
d = \frac{p}{r+1}.
\]
Consequently, the primitive $\phi_p$ has the form (\ref{psi 1 LMNS}).
\end{proof}

\medskip\noindent
We now examine the zero-inflated property. We use similar notations as in
the ABM case and denote by 
$P_r(0; p,m)$ the zero-mass probability of the distribution
associated with the LM NEF obtained by VF $m/(1-\frac{m}{p})^{r}$.

\begin{proposition}\label{p:lmzero}
The zero-mass probability $P_{r}(0;p,m)$
in the LM class is increasing in $r\geq 0$.
\end{proposition}

\begin{proof}
Similar as in the proof of Proposition \ref{p:abmzero}
it suffices to show that $\phi_{1,r}(m)$ is decreasing, where
$\phi_{1,r}(m)$ is given in \eqref{psi 1 LMNS} with $p=1$:
\[
\phi_{1,r}(m) = \frac{1}{r+1}\big(1-(1-m\big)^{r+1}\big),
\quad 0<m<1;\; r=0,1,\ldots.
\]
Define functions $\{f_x(s) : (0,\infty)\to\mathbb{R}, x\in(0,1)\}$ by
\[
f_x(s) = \frac1s(1-x^s).
\]
In the proof of Proposition \ref{p:abmzero} we showed
that $f_x(s)$ is decreasing (in $s>0$)
for any $x\in (0,1)$. Specifically, we get that
$\phi_{1,r}(m)$ is decreasing for $r\geq 0$, and 
$P_{r}(0;p,m)$ is increasing for $r\geq 0$ for any values of 
$p>0$ and $m\in(0,p)$.

\end{proof}

\begin{remark}\label{r:lmcont}
For the LM class we assumed that $r$ is a natural number. However, all
results obtained for this class are also correct for any real number 
$r\geq 1$ as the LM class of VFs can be shown to fulfill the premises of
\citet[Proposition 4.4]{letac1990}. 
Consequently, the finite sum in 
(\ref{psi LMNS}) could be replaced by sum of entire series 
using the binomial series of
Newton instead of the binomial formula of Pascal. 
Note that the proof of Proposition \ref{p:lmzero}
used already any $r\geq 1$.

\noindent
However, we focus in this work only on classes for which we can obtain
relatively simple expressions for both $\psi_p(m)$ and $\phi_p(m)$ in the
form of finite sums and the like and not in sums of entire series. This is
the reason why we have excluded the HD class (see (\ref{HD1})) from further
consideration.
\end{remark}


\section{A Numerical Example}\label{s:num}

In this section we show that our classes are very well suited for fitting
small counting data. Consider the well-known 6 data sets of automobile
insurance claims per policy over a fixed period of time that have been
studied in \citet{gossiaux1981}.
They fitted Poisson (P) and the
negative Binomial distributions (NB),
Since then, many  models have been developed for fitting one or more
of these data sets
\citep{willmot1987,ruohonen1988,denuit1997,kokonendji2004a,
kokonendji2004b,gomez2011a,gomez2011b,
gencturk2016,gomez2016,bhati2019,castellares2020}.
It is not the purpose of this paper to give
a full description of all the 6 data sets, of all these fitting models,
and of a full comparison with our ABM and LM models.
For the complete picture, we refer to the ancillary file
of \citet{barlev2020}. 
Here, we consider the data set
of insurance claims in Zaire in 1974.

\begin{table}[H]
\caption{Data set of insurance claims.}
\label{t:dataset}
\medskip \centering
\begin{tabular}{||l|c|c|c|c|c|c||}
\hline
value & 0 & 1 & 2 & 3 & 4 & 5 \\
\hline
frequency & 3719 & 232 & 38 & 7 & 3 & 1\\
\hline
\end{tabular}
\end{table}

\noindent
The descriptive statistics of the data 
show over-dispersion
(index of dispersion, defined as the variance
divided by the mean, is larger than one),
zero-inflation (fraction of zeros is more than 90\%), 
and relative large skewness and kurtosis.

\begin{table}[H]
\caption{Descriptive statistics.}
\label{t:statistics}
\medskip \centering
\begin{tabular}{||l|l||}
\hline
number of observations & $4000$\\
mean     & $0.0865000$ \\
variance & $0.122548$\\
skewness & $5.31602$ \\
kurtosis & $41.0067$\\
fraction zeros & $0.929750$ \\
index of dispersion & $1.41674$\\
\hline
\end{tabular}
\end{table}

\noindent
The following models and their probability mass functions
have been considered for fitting (next to
the default Poisson and negative binomial).

\begin{enumerate}[(a).]
\item
PIG (Poisson-inverse Gaussian distribution) in
\citet{willmot1987},
\[
p_n = \int_0^\infty \frac{e^{-x}\,x^n}{n!}\, f(x;\beta,\mu)\,dx, 
\; n=0,1,\ldots,
\]
where $f(x;\beta,\mu)$ is the inverse Gaussion pdf of the form
\[
f(x;\beta,\mu) = \frac{\mu}{\sqrt{2\pi\beta x^3}}\,
e^{-\frac{(x-\mu)^2}{2\beta x}},\; x>0,
\]
with positive parameters $\beta$ and $\mu$.

\item
PGD (Poisson-Goncharov distribution) in \citet{denuit1997},
\[
p_n = G_n e^{u_n}, \; n=0,1,\ldots,
\]
where $u_n=\theta_1+\theta_2n\theta_3n^2$ for three real-valued
parameters $\theta,\theta_2,\theta_3$, and where the $G_n$'s are
defined recursively by
\[
G_0=1;\;G_n=-\sum_{i=0}^{n-1}\frac{u_i^{n-i}}{(n-i)!}\,G_i.
\]

\item
DLD (discrete Lindley distribution) in \citet{gomez2011a}, 
\[
p_n = \frac{\lambda^n}{1-\log\lambda}\,
\big(\lambda\log\lambda + (1-\lambda)(1-(n+1)\log\lambda)\big)
\; n=0,1,\ldots,
\]
with parameter $0<\lambda<1$.

\item
NLD (new logarithmic distribution) in \citet{gomez2011b},
\[
p_n = \frac{\log(1-\alpha\theta^n) - \log(1\alpha\theta^{n+1}}%
{\log(1-\alpha)}\; n=0,1,\ldots,
\]
with parameters $\alpha<1$ ($\alpha\neq 0$), and $0<\theta<1$.

\item
PLB (Poisson-Lindley-Beta prime distribution) in \citet{gomez2016},
\[
p_n = \frac{\alpha(1+\alpha)\Gamma(\alpha+\beta)\Gamma(\beta+n)}%
{\Gamma(\beta)\Gamma(\alpha+\beta+n+3)}\,
\big((\beta+n)(2+n) + \alpha+2\big)\; n=0,1,\ldots,
\]
with positive parameters $\alpha,\beta$.

\item
GDP (a new geometric discrete Pareto distribution) in \citet{bhati2019},
\[
p_n= \frac{q^n}{(n+1)^\alpha} - \frac{q^{n+1}}{(n+2)^\alpha},
\; n=0,1,\ldots,
\]
with parameters $0<q\leq 1$ and $\alpha\geq 0$.

\item
BTD (Bell-Touchard discrete distribution) in \citet{castellares2020},
\[
p_n = \frac1{n!}\,e^{\theta(1-e^\alpha)}\, \alpha^n T_n(\theta),
\; n=0,1,\ldots,
\]
with positive parameters $\alpha,\theta$, and where the $T_n(\cdot)$
are Touchard polynomials; i.e,
\[
T_n(\theta) = e^{-\theta}\sum_{k=0}^\infty \frac{k^n\,\theta^k}{k!}
\]
\end{enumerate}

\medskip\noindent
It would be interested to include in our comparison study
the EDM of the Poisson-Tweedie class \citep{kokonendji2004b,jorgensen2016}.
However, we decided to leave out this model because the numerical application
to fitting data is in these papers not given in its generality
but only for the specific case of Poisson-inverse Gaussian distribution.
It is outside the scope of this paper to develop a full numerical
procedure for computing the distributions of the 
Poisson-Tweedie class, which will be exploited in a subsequent study.  

We compare the fitted distributions of the models given above
with the distributions of our ABM and LM classes
for a range of powers $r$. 
As a fair comparison we consider only two-parameter distributions.
The one-parameter discrete Lindley distribution gives a bad fit, 
the three-parameter  Poisson-Goncharov distribution
would give an almost perfect fit.
The performances of the fitted remaining distributions (a), (d)-(g) 
are computed, using the parameters that
are reported in the cited references,
whereas the parameters $p$ and $m$ of our
models are computed by maximum
likelihood estimation.

The computation of the probabilities \eqref{probability of X}
is done by a numerical computer program. Given parameters
$p$ and $m$ of the variance function
$V_p(m)$, the functions $\psi_p$ and $\phi_p$ follow
from Proposition \ref{p:abmfunctions} (ABM class)
and Proposition \ref{p:lmfunctions} (LM class),
the measure $(\mu_p^*)_n$ is computed 
numerically by solving the derivatives in \eqref{LM1}.
For more details on the implementations for the ABM and LM classes
we refer to \cite{barlev2020}. The resulting estimated parameters 
are for $r=1,\ldots,10$

\begin{table}[H]
\caption{ABM and LM Parameters.}
\label{t:parms}
\medskip \centering
\begin{tabular}{||r|ll|ll||}
\hline
 & \multicolumn{2}{c|}{ABM} & \multicolumn{2}{c||}{LM}\\
\multicolumn{1}{||c|}{$r$} & \multicolumn{1}{c}{$m$} & 
\multicolumn{1}{c}{$p$} & \multicolumn{1}{|c}{$m$} & 
\multicolumn{1}{c||}{$p$} \\
\hline
1 & 0.086500 & 0.216600 & 0.086500 & 0.277098\\
2 & 0.086500 & 0.459964 & 0.086500 & 0.520502\\
3 & 0.086500 & 0.704120 & 0.086500 & 0.764666\\
4 & 0.086500 & 0.948471 & 0.086500 & 1.009018\\
5 & 0.086500 & 1.192899 & 0.086500 & 1.253448\\
6 & 0.086500 & 1.437365 & 0.086500 & 1.497914\\
7 & 0.086500 & 1.681853 & 0.086500 & 1.742403\\
8 & 0.086500 & 1.926354 & 0.086500 & 1.986905\\
9 & 0.086500 & 2.170867 & 0.086500 & 2.231417\\
10 & 0.086500 & 2.415385 & 0.086500 & 2.475934\\
\hline
\end{tabular}
\end{table}

\medskip\noindent 
Let $x_0, x_1,\ldots,x_K$ be the data set of counts as presented in
Table \ref{t:dataset}; thus $K=6$, and 
$x_k$ is the observed number of value $k$.
Let $N=\sum_{k=0}^Kx_k$ be the total number of observations. The empirical
probability mass function is 
\begin{equation*}
p^{\mathrm{(emp)}}_k=\frac{x_k}{N},\;k=0,\ldots,K.
\end{equation*}
The performance of a fitting model is expressed through the following
measures.

\begin{itemize}

\item 
$\chi^2$ value; taken into account sufficient expected number in the
categories.

\item $p$-value of the $\chi^2$ quantile; taken into account the number of
parameters that are estimated from the data.

\item Root mean squared error (RMSE): 
\begin{equation*}
\sqrt{\frac{1}{K+1}\sum_{k=0}^K \big(x_k - Np_k^{\mathrm{mod}}\big)^2}.
\end{equation*}

\item
Kullback-Leibler divergence (KL) :
\[
\sum_{k=0}^K p_k^{\mathrm(emp)}\,\log\frac{p_k^{\mathrm(emp)}}{p_k^{\mathrm(mod)}}.
\]
\end{itemize}

\medskip\noindent 

\begin{table}[H]
\caption{Performance measures of the fitting models.
ABM and LM for $r\in\{1,\ldots,10\}$ that gave the 
highest $p$-value.}
\label{t:measures}
\medskip \centering
\begin{tabular}{||l|rrrrr||}
\hline
\multicolumn{1}{|c|}{model} & \multicolumn{1}{c}{$\chi^2$} & 
\multicolumn{1}{c}{df} & \multicolumn{1}{c}{$p$-value} 
& \multicolumn{1}{c}{RMSE} & \multicolumn{1}{c|}{KL} \\ 
\hline
PIG & 0.543789 & 2 & 0.761935 & 1.760396 & 1.9866e-04\\
NLD & 2.312184 & 2 & 0.314714 & 2.235378 & 3.1047e-04\\
PLB & 0.370556 & 2 & 0.830873 & 1.488770 & 2.0715e-04\\
GDP & 0.383445 & 2 & 0.825536 & 1.303073 & 1.7734e-04\\
BTD & 9.251567 & 2 & 0.009796 & 4.056789 & 8.6835e-04 \\
\hline
ABM($r=10$)& 0.444362 & 2 & 0.800770 & 0.726298 & 1.5896e-04\\
LM($r=4$)  & 0.382901 & 2 & 0.825760 & 1.044667 & 1.6972e-04\\
\hline
\end{tabular}
\end{table}

\medskip\noindent 
When we consider the $p$-value criterion, 
several distributions are compatible, including the ABM and LM models.
However, for the RMSE and Kullback-Leibler criteria,
the ABM and LM models show a major improvement. 
In \citet{barlev2020} we give an overview of comparisons of
these models for many more data sets having various different 
statistical properties. The overall picture  is
that our models show competitive, or best
performances in all cases, while the other models 
perform sometimes good sometimes bad. 

\medskip\noindent
It might be of interest to present the performances of
the investigated ABM and LM classes, $r=1,\ldots,10$. We show
these in figures.

\begin{figure}[H]
\centering
\includegraphics[width=0.32\textwidth]{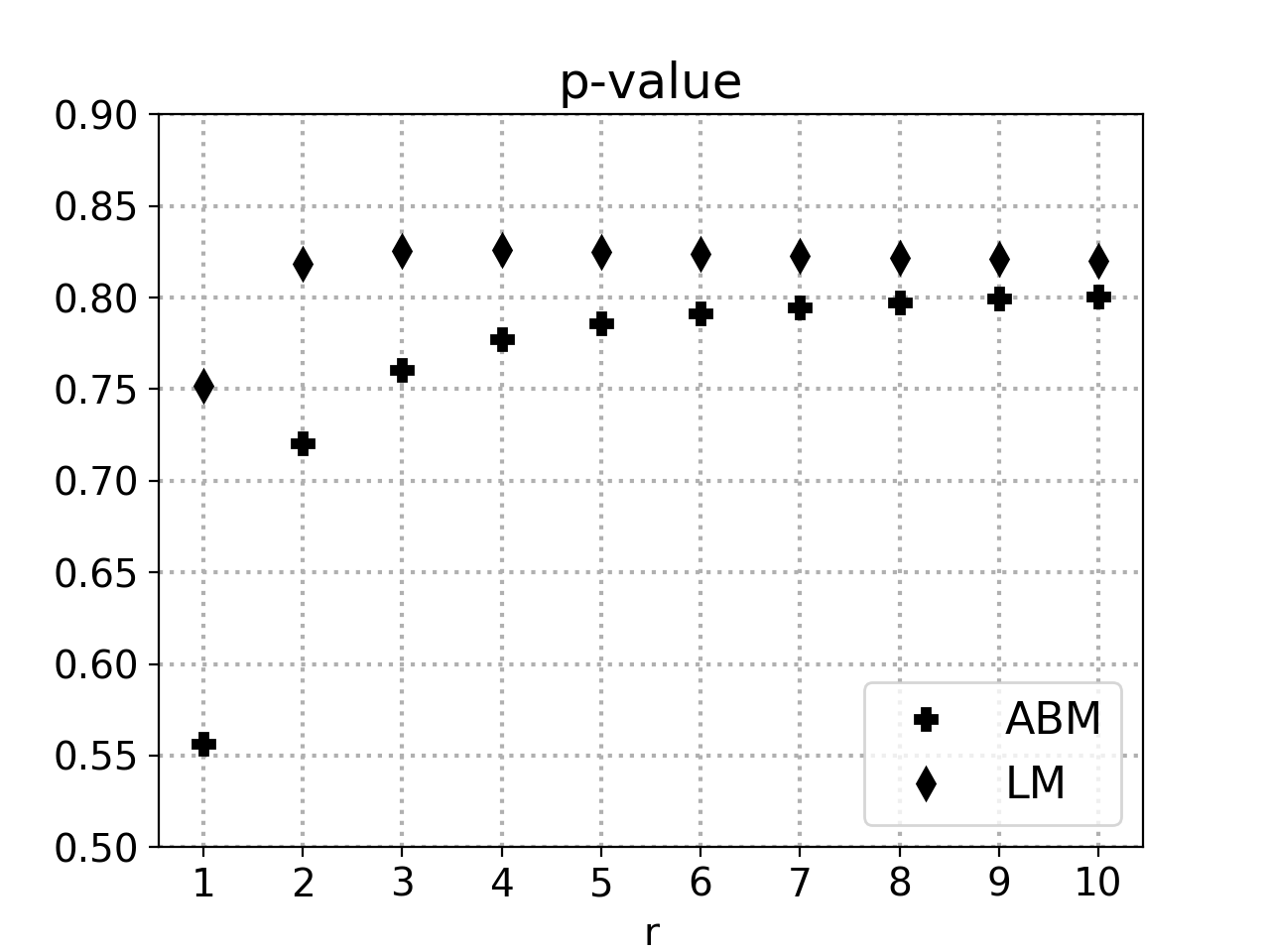}
\includegraphics[width=0.32\textwidth]{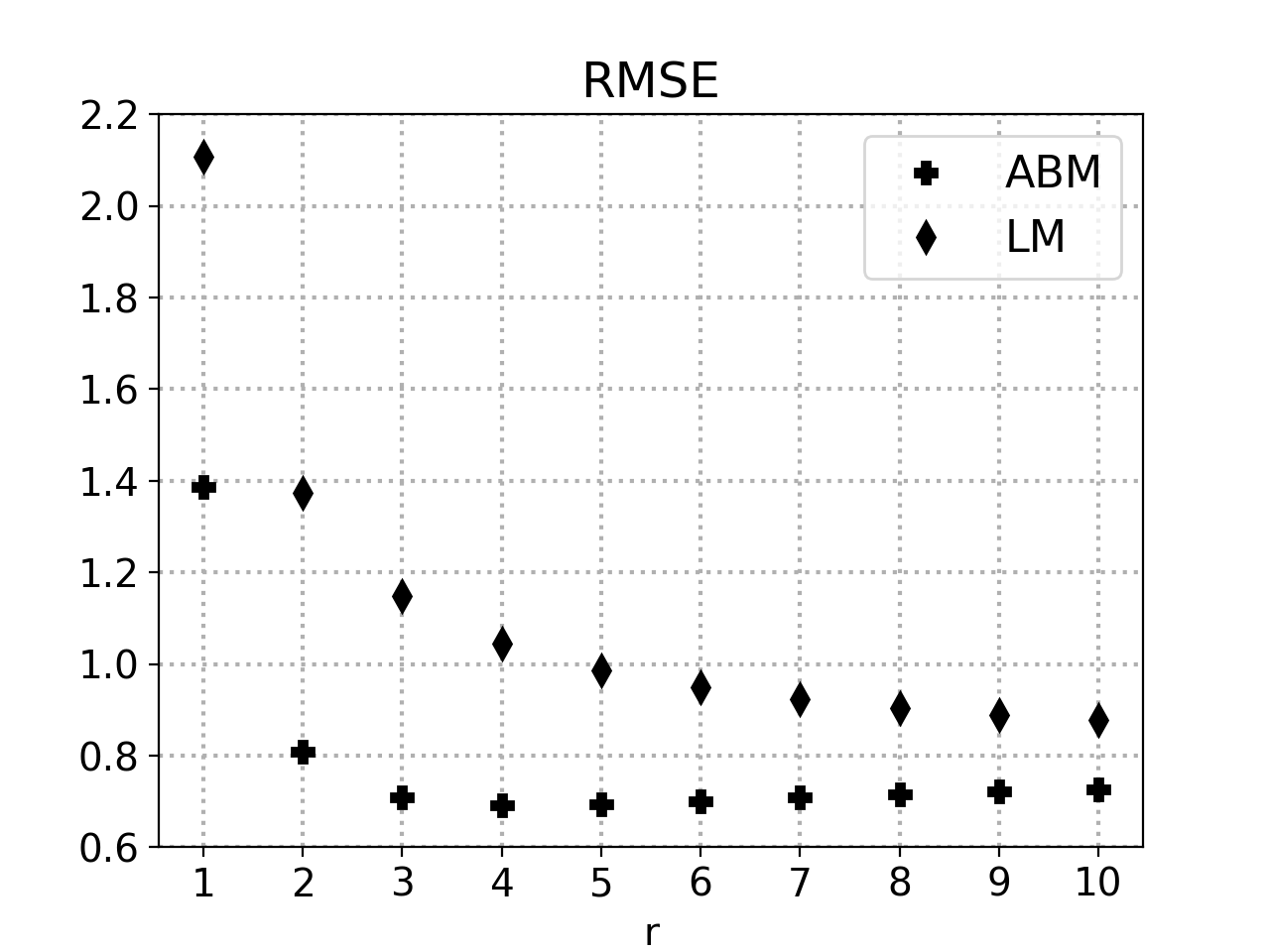}
\includegraphics[width=0.32\textwidth]{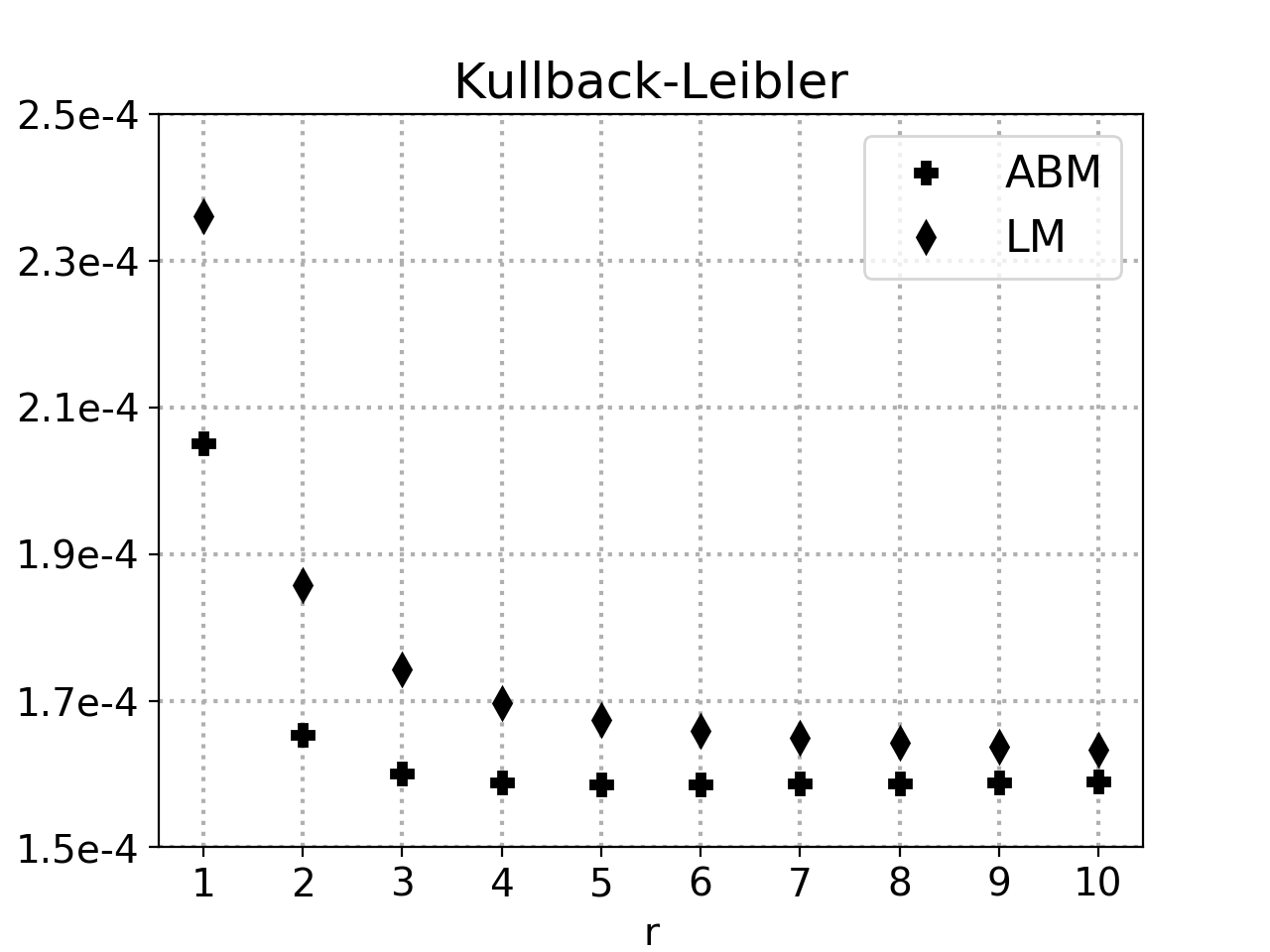}
\caption{Performances of the fitted models of the ABM and LM classes for $r$ 
upto 10.}
\label{f:performancefigs}
\end{figure}

\noindent
From these figures we make a few observations. Firstly, the LM class performs
better than the ABM class for the $p$-value performance criterion, but not
for the other criteria. Secondly, when the degree $r$ of the variance function
increases, the performances of ABM and LM become more and more equal.
Also, note that performance criteria do not alsway show a monotone behaviour.

\section{Concluding Remarks}\label{s:conclusion}

\begin{enumerate}
\item 
In this paper we have attempted in exposing 'new' EDMs of distribution
supported on the set of nonnegative integers. Such EDMs can be represented
only by their mean value parametrization 
whereas their respective generating
measure can be computed via (\ref{LM1}) by existing powerful mathematical
software. The expressions obtained for the $\mu_{n}$'s  will depend, of
course, on the unknown dispersion 
parameter $p$. 
Based on a random sample,
the MLE is the sample mean whereas the parameter $p$ 
can be
estimated by the method of moments estimation. All that is said above
depends, of course, on the ability to locate classes of VFs of the 
form (\ref{LM}) or (\ref{LM2}) for which both the $\psi$ and the $\phi$ 
functions possess explicit
and 'nice expressions in terms of $m$. In such a case the likelihood
function is well expressed, a fact that has a tremendous significance in
statistics. Obviously, if such EDMs are used in a Bayesian framework there
is no need to compute the $\mu_{n}$'s.

\item 
From the presentation of these two classes, it will be easy to see
that more classes of the same type (i.e., subclasses of either 
(\ref{LM}) or (\ref{LM2})) can be constructed. 
However, we will suffice with presenting
only the two classes.

\item 
The classes of EDMs introduced in this paper can be used, for example,
as competitors and alternatives to the Poisson or negative binomial NEFs for
modeling count data in various actuarial aspects and insurance claims. This
has been indeed demonstrated in the numerical section. However, based on our
experience in the insurance and actuarial industry, we have noticed that
professionals are very concerned about using new (both discrete and
continuous) distributions to estimate and evaluate various relevant
parameters as the insurance risk factor. So in another paper of ours
\citep{barlev2019} we considered, just for the sake of
demonstration, the problem of computing the insurance risk factor 
\begin{equation*}
\ell (x)=\mathbb{P}\big(\sum_{k=0}^{N}Y_{k}>x\big),
\end{equation*}%
for large values of $x$, where $N$ is a discrete random variable,
counting the number of claims
during a fixed period of time, and the $Y_{i}$'s are the respective
independent claim sizes. The conventional actuarial literature is full with
models in which $N$ has either Poisson or negative binomial distributions
whereas the $Y_{i}$'s have a common gamma or inverse Gaussian or even
positive stable distribution. \citet{barlev2019} used
'unconventional' NEF distributions for $N$ by taking the Abel, strict
arcsine and Tak\'{a}cs NEFs 
\citep[i.e., NEFs having cubic VFs characterized by][]{letac1990}.
For data of a Swedish claims at a car insurance company they
considered all combinations of the distributions of $N$ and the $Y_{i}$'s
mentioned above and demonstrated that the best fit for such data is obtained
for the pair (arcsine, positive stable) with $p$-value equals $.7460$. All
fit ranking after are, respectively, (arcsine, inverse Gaussian, $p$-value 
$0.4224$), (Tak\'{a}cs, gamma, $p$-value $0.4159$), (Abel, positive stable, 
$p $-value $0.3089$), (Tak\'{a}cs, inverse Gaussian, $p$-value $0.2800$), 
(Tak\'{a}cs, positive stable, $p$-value $0.2701$), 
(Abel, inverse Gaussian, $p$-value $0.2459$) and 
(Abel, gamma, $p$-value $0.2101$). As opposed to these,
the worst fit has been obtained for pairs of the Poisson along with the
gamma, inverse Gaussian and positive stable distributions with $p$-value
less than $.00001$.

\item 
Consequently, we trust that the ABM and and LM classes (as well as other
similar classes) are going to play a significant role as a 'new generation'
of counting distributions and to have a 'prosperous future' in applications
to actuarial science data as well as to other statistical data. Indeed, the
present authors \citep[see][]{barlev2020} conducted a project in which
more than 20 sets of count data from 
the statistical literature were collected. Such
data were modeled by some conventional discrete distributions and were
compared to count probabilities belonging to the ABM and LM classes.
And so, as we expected, all of the latter count probabilities, and with
respect to various of metrics or goodness-of-fit tests, have shown
superiority and provided a much better fit for each of these 
data sets.

\item 
One last remark. Researchers may avoid using the LM class as it is
non-steep. However, another important class of NEFs having power VFs of the
form $(V,M)=(\alpha m^{\gamma},\mathbb{R}^{+}),\alpha >0,\gamma <0$ 
(which belong to the Tweedie scale) is also non-steep. 
Indeed, for the latter class 
$M=\mathbb{R}^{+}$ whereas its convex support $C=\mathbb{R}$. This class
though is frequently used in various applications.
\end{enumerate}

\bigskip\noindent
\textbf{Acknowledgements}. 
The authors are 
indebted to two referees for their careful reading, criticism,
and constructive comments which significantly improved the 
presentation of of the paper. We are also
extremely grateful to G\'{e}rard
Letac for his careful reading of the previous versions of the paper and for
his wise, constructive and helpful comments which resulted in much improved
version. The part of work of Shaul Bar-Lev is partially supported by the
Netherlands Organization for Scientific Research (NWO) project number
040.11.711.

\appendix

\section{LM Miscellaneous}\label{a:lmmisc}
We now present a second way to compute the $\mu _{n}$'s for the NEFs
corresponding to the LM class by means of Hermite polynomials, a way
suggested to us by G\'{e}rard Letac (a personal communication).

\begin{proposition}\label{p:numeasure}
Let $\mathcal{F}$ be the NEF corresponding to the VF 
$m/(1-\frac{m}{p})^{r}, 0<m<p$. 
Then there exists a positive measure $\nu $ on the set of
positive integers $\mathbb{N}$ such that $\mathcal{F}$ is generated by
\begin{equation*}
\mu =e^{p\nu }=\delta _{0}+\sum\nolimits_{k=1}^{\infty }
\frac{p^{k}}{k!}\nu^{\ast k},
\end{equation*}
where $\nu^{\ast k}$ is the $k$-th fold convolution of $\nu $,
\begin{equation*}
\nu(n)=\frac{1}{n!n}\left[ 
\left( \frac{d}{dm}\right)^{n-1}e^{nP(m)}\right]_{m=0}
\end{equation*}
and
\begin{equation*}
P(m)=-\sum_{k=1}^{\infty }\frac{(-r)_{k}}{k!k}m^{k},
\end{equation*}
where $(-r)_{k}$ is the Pochhammer symbol
\begin{equation*}
(-r)_{k}=-r(-r+1)(-r+2)\cdot \cdot \cdot (-r+k-1).
\end{equation*}
\end{proposition}

\begin{proof}
It sufficers to consider $p=1$.
For this case we have
\begin{equation*}
d\theta =\frac{dm}{V_{\mathcal{F}}(m)}=(1-m)^{r}\frac{dm}{m}=\frac{dm}{m}
+ \sum_{k=1}^{\infty }\frac{(-r)_{k}}{k!}m^{k-1}dm
\end{equation*}
Thus $\theta =\log m-P(m)$ which by denoting $w=e^{\theta }$ we get 
$m=we^{P(m)}$. 
Now apply the Lagrange formula which states that if 
$h(w)=wg(h(w))$ then
\begin{equation*}
h(w)=\sum_{n=1}^{\infty }\frac{w^{n}}{n!}\left[ 
(\frac{d}{dm})^{n-1}(g(m))^{n}\right]_{m=0}.
\end{equation*}
When applying this formula to $m=h(w)=k_{\mu }^{\prime }(\theta )$ and 
$g(m)=e^{P(m)}$ we get
\begin{equation*}
k_{\mu }^{\prime }(\theta )=\sum\nolimits_{n=1}^{\infty }\frac{w^{n}}{n!}
\left[ (\frac{d}{dm})^{n-1}e^{nP(m)}\right] _{m=0}.
\end{equation*}
Since $d\theta =dw/w$ we obtain
\begin{equation*}
k_{\mu }(\theta )=\sum\nolimits_{n=1}^{\infty }\frac{w^{n}}{n!n}
\left[ (\frac{d}{dm})^{n-1}e^{nP(m)}\right] _{m=0}
=\sum_{n=1}^{\infty }\nu(n)w^{n},
\end{equation*}
and the remainder of the proof is standard.
\end{proof}

\begin{example}
For $r=1$, $P(m)=m$ and $\nu (n)=n^{n-2}/n!.$
\end{example}

\begin{example}
For $r=2$, $P(m)=2m-m^{2}/2$ but the computation of%
\begin{equation*}
\left[ (\frac{d}{dm})^{n-1}e^{n(2m-m^{2}/2)}\right] _{m=0}
\end{equation*}
is more delicate. For such a computation with use the formula for Hermite
polynomials \citep[see][p.130]{rainville1960},
by which
\begin{equation*}
e^{2xt-t^{2}}=\sum_{k=0}^{\infty }H_{k}(x)\frac{t^{k}}{k!}.
\end{equation*}
Setting $x=\sqrt{2n}$ and $t=\sqrt{n/2}m$ yields
\begin{equation*}
e^{n\left( 2m-m^{2}/2\right) }
=\sum_{k=0}^{\infty }H_{k}(\sqrt{2n})
\left( \frac{n}{2}\right) ^{k/2}\frac{m^{k}}{k!}.
\end{equation*}
By employing the Taylor formula it follows that
\begin{equation*}
\left[ (\frac{d}{dm})^{n-1}e^{n(2m-m^{2}/2)}\right]_{m=0}
=H_{n-1}(\sqrt{2n})\left( \frac{n}{2}\right) ^{(n-1)/2}
\end{equation*}
and thus
\begin{equation*}
\nu (n)=\frac{1}{n!n}H_{n-1}(\sqrt{2n})\left( \frac{n}{2}\right) ^{(n-1)/2}.
\end{equation*}
\end{example}


\begin{thebibliography}{99}

\bibitem[Abid, Kokonendji and Masmoudi(2020)]{abid2020}
Abid, R., Kokonendji, C.C. and Masmoudi, A. (2020).
On Poisson-exponenital-Tweedie models for ultra-overdispersed
count data.
\textit{AStA Advances in Statistical Analysis}, available online

\texttt{doi.org/10.1007/s10182-020-00375-4}

\bibitem[Awad, Bar-Lev and Makov(2016)]{awad2016} 
Awad, Y., Bar-Lev, S.K. and Makov, U. (2016). 
A new class counting distributions embedded in the Lee-Carter model
for mortality projections: A Bayesian approach. A technical report No. 146,
Actuarial Research Center, University of Haifa, Israel.

\bibitem[Bar-Lev(1987)]{barlev1987} 
Bar-Lev, S.K. (1987). 
Discussion on paper by B. J\o rgensen, 
``Exponential dispersion models''. 
\textit{Journal of the Royal Statistical Society Series B} 49(2), 153-154.

\bibitem[Bar-Lev et al(1992)]{barlev1992b} 
Bar-Lev, S.K., Bshouty, D., Enis, P. and Ohayon, A.Y. (1992). 
Compositions and products of infinitely divisible variance functions. 
\textit{Scandinavian Journal of Statistics} 19(1), 83-89.

\bibitem[Bar-Lev and Enis(1986)]{barlev1986} 
Bar-Lev, S. K. and Enis, P. (1986).
Reproducibility and natural exponential families with power variance
functions. 
\textit{The Annals of Statistics} 14(4), 1507-1522.

\bibitem[Bar-Lev and Kokonendji(2017)]{barlev2017} 
Bar-Lev, S.K. and Kokonendji, C.C. (2017). 
On the mean value parameterization of natural exponential families
-- a Revisited Review. 
\textit{Mathematical Methods of Statistics} 26(3), 159-175.

\bibitem[Bar-Lev and Ridder(2019)]{barlev2019} 
Bar-Lev, S.K. and Ridder, A. (2019). 
Monte Carlo methods for insurance risk computation. 
\textit{International Journal
of Statistics and Probability}, 8(3), 54-74.

\bibitem[Bar-Lev and Ridder(2020)]{barlev2020} 
Bar-Lev, S.K. and Ridder, A. (2020).
Exponential Dispersion Models for Overdispersed Zero-Inflated Count Data.
arXiv: 2003.13854v1 [stat.ME] 30 Mar 2020. With ancillary file.

\bibitem[Barndorff-Nielsen(1978)]{barndorff1978} 
Barndorff-Nielsen, O. (1978). 
\textit{Information and Exponential Families in Statistical Theory}. 
Wiley, New York.

\bibitem[Bhati and Bakouch(2019)]{bhati2019}
Bhati, D., and H.S. Bakouch (2019).
A new infinitely divisible discrete distribution with applications 
to count data modelling.
\textit{Communications in Statistics - Theory and Methods} 48(6), 
1401-1416.


\bibitem[Bryc and Ismail(2005)]{bryc2005} 
Bryc, W. and Ismail, M. (2005). 
Approximation operators, q-exponential, and free exponential families. Preprint.
Available as arXiv:math/0512224.

\bibitem[Castellares, Lemonte and  Moreno–Arenas(2020)]{castellares2020}
Castellares, F., A. J. Lemonte, and  G. Moreno–Arenas (2020).
On the two-parameter Bell–Touchard discrete distribution.
\textit{Communications in Statistics - Theory and Methods}
49(19), 4834-4852.

\bibitem[Consul(1989)]{consul1989} 
Consul, P.C. (1989). \textit{Generalized Poisson
Distributions: Properties and Applications}. Marcel Dekker, New York.

\bibitem[Consul and Famoye(2006)]{consul2006} 
Consul, P.C. and Famoye, F. (2006). 
\textit{Lagrangian Probability Distributions}. 
Birkh\"{a}user, Boston, Basel, Berlin.

\bibitem[Denuit(1997)]{denuit1997}
Denuit, M. (1997). 
A New Distribution of Poisson-Type for the Number of Claims. 
\textit{ASTIN Bulletin} 27(2), 229-242. 

\bibitem[Gen\c{c}t\"urk and Yi\v{g}iter(2016)]{gencturk2016}
Gen\c{c}t\"urk and A. Yi\v{g}iter (2016).
Modelling claim number using a new mixture model: negative binomial 
gamma distribution.
\textit{Journal of Statistical Computation and Simulation} 
86(10), 1829-1839.

\bibitem[G\'omez-D\'eniz and Calderin-Ojeda(2011)]{gomez2011a} 
G\'omez-D\'eniz, E. and Calderin-Ojeda, E. (2011). 
The discrete Lindley distribution: properties and applications. 
\textit{Journal of Statistical Computation and Simulation} 81(11), 
1405-1416.

\bibitem[G\'omez-D\'eniz, Sarabia and Calderin-Ojeda(2011)]{gomez2011b}
G\'omez-D\'eniz, E., J.M. Sarabia, and E. Calderin-Ojeda (2011).
A new discrete distribution with actuarial applications.
\textit{Insurance: Mathematics and Economics} 48(3), 406-412.

\bibitem[G\'omez-D\'eniz, Hern\'andez-Bastida and 
Fern\'andez-S\'anchez(2016)]{gomez2016}
G\'omez-D\'eniz, E., Hern\'andez-Bastida, A., and Fern\'andez-S\'anchez, 
M.P.A. (2016).
A Suitable Discrete Distribution for Modelling Automobile Claim Frequencies.
\textit{Bulletin of the Malaysian Mathematical Sciences Society}
39, 633-647.

\bibitem[Gossiaux and Lemaire(1981)]{gossiaux1981}
Gossiaux, A., and J. Lemaire (1981). 
Methodes d'ajustement de distributions de sinistres. 
\textit{Bulletin of the Association of Swiss Actuaries} 81, 87-95.

\bibitem[Hinde and Dem\'etrio(1998)]{hinde1998} 
Hinde, J. and Dem\'{e}trio, C.G.B. (1998).
Overdispersion: Models and Estimation. 
\textit{Computational Statistics and Data Analysis} 27, 151-170.


\bibitem[J\o rgensen(1987)]{jorgensen1987} 
J\o rgensen, B. (1987). 
Exponential dispersion models (with discussion), 
\textit{Journal of the Royal Statistical Society, Ser. B} 49(2), 127-162.

\bibitem[J\o rgensen(1997)]{jorgensen1997} 
J\o rgensen, B. (1997). 
\textit{The Theory of Exponential Dispersion Models}, 
Monographs on Statistics and Probability,
Vol. 76, Chapman and Hall, London.

\bibitem[J\o rgensen and Kokonendji(2016)]{jorgensen2016} 
J\o rgensen, B. and C.C. Kokonendji (2016). 
Discrete dispersion models and their Tweedie asymptotics.
\textit{AStA Advances in Statistical Analysis} 100, 43-78.

\bibitem[Kokonendji, Dossou-Gb\'{e}t\'{e} and Dem\'{e}trio(2004)]{kokonendji2004b} 
Kokonendji, C.C., Dossou-Gb\'{e}t\'{e}, S. and Dem\'{e}trio, C.G.B. (2004). 
Some discrete exponential dispersion models: 
Poisson-Tweedie and Hinde-Dem\'{e}trio classes. 
\textit{Statistics and Operations Research Transactions} 28(2),
201-214.

\bibitem[Kokonendji, Dem\'{e}trio and Zocchi(2007)]{kokonendji2007}
Kokonendji, C.C., Dem\'{e}trio, C.G.B. and Zocchi, S.S. (2007). 
On Hinde--Dem\'{e}trio regression models for overdispered count data.
\textit{Statistical Methodology} 4, 277-291.

\bibitem[Kokonendji and Khoudar(2004)]{kokonendji2004a} 
Kokonendji, C.C. and Khoudar, M. (2004). 
On strict arcsine distribution. 
\textit{Communications in Statistics - Theory and Methods} 33(5),
993-1006.


\bibitem[Kokonendji and Malouche(2008)]{kokonendji2008} 
Kokonendji, C.C. and Malouche, D. (2008). 
A property of count distributions in the Hinde-Dem\'{e}trio family. 
\textit{Communications in Statistics - Theory and Methods} 37(12), 1823-1834.

\bibitem[Lee and Carter(1992)]{lee1992} 
Lee, R.D. and Carter, L. (1992). 
Modelling and forecasting the time series of US mortality. 
\textit{Journal of the American Statistical Association} 87(419), 659-671.

\bibitem[Letac and Mora(1990)]{letac1990} 
Letac, G. and Mora, M. (1990). 
Natural real exponential families with cubic variance functions. 
\textit{The Annals of Statistics} 18(1), 1-37.

\bibitem[Morris(1982)]{morris1982} 
Morris, C. N. (1982). 
Natural exponential families with quadratic variance functions. 
\textit{The Annals of Statistics} 10(1), 65-80.

\bibitem[Rainville(1960)]{rainville1960} 
Rainville, E.D. (1960). \textit{Special Functions}.
The Macmillan Company, New York.

\bibitem[Renshaw and Haberman(2006)]{renshaw2006} 
Renshaw, A.E. and Haberman, S. (2006). 
A cohort-based extension to the Lee-Carter model for mortality reduction factors. 
\textit{Insurance: Mathematics and Economics} 38(3), 556-570.

\bibitem[Ruohonen(1988)]{ruohonen1988}
Ruohonen, M. (1988). 
On A Model for the Claim Number Process. 
\textit{ASTIN Bulletin} 18(1), 57-68. 

\bibitem[Tweedie(1984)]{tweedie1984} 
Tweedie, M. C. K. (1984). An index which
distinguishes between some important exponential families. In Statistics:
Applications and New Directions. Proc. Indian Institute Golden Jubilee
lnternat. Conf. (J. K. Ghosh and J. Roy, eds.) 579-604. Indian Statist.
Inst., Calcutta.

\bibitem[Willmot(1987)]{willmot1987}
Willmot, G. (1987). 
The Poisson-inverse Gaussian distribution as an alternative 
to the negative binomial. 
\textit{Scandinavian Actuarial Journal} 3-4, 113-127.

\end{thebibliography}
\end{document}